\documentclass[leqno]{amsart}
\usepackage{amsmath,amsfonts,amsthm,amssymb,indentfirst,epic,url,graphics}

\newtheorem{theorem}{Theorem}
\newtheorem{lemma}[theorem]{Lemma}
\newtheorem{corollary}[theorem]{Corollary}
\newtheorem{proposition}[theorem]{Proposition}
\newtheorem{definition}[theorem]{Definition}
\newtheorem{question}[theorem]{Question}

\newcommand{\Z}{\mathbb Z}
\newcommand{\Q}{\mathbb Q}
\newcommand{\R}{\mathbb R}
\renewcommand{\r}{\mathrm}
\newcommand{\mz}{'} 

\newcommand{\langl}{\begin{picture}(5.1,10)
\put(1.1,3.3){\rotatebox{60}{\line(1,0){5.5}}}
\put(1.1,3.3){\rotatebox{300}{\line(1,0){5.5}}}
\end{picture}}

\newcommand{\rangl}{\begin{picture}(5,10)
\put(.9,3.3){\rotatebox{120}{\line(1,0){5.5}}}
\put(.9,3.3){\rotatebox{240}{\line(1,0){5.5}}}
\end{picture}}

\DeclareRobustCommand{\smallcoprod}{\begin{picture}(10,6)
\thicklines
\put(1.8,.6){\line(1,0){6.6}}
\put(3.6,.6){\line(0,1){6}}
\put(6.6,.6){\line(0,1){6}}
\end{picture}}

\begin{document}

\begin{center}
\texttt{Comments, corrections,
and related references welcomed, as always!}\\[.5em]
{\TeX}ed \today
\vspace{2em}
\end{center}

\title%
{Some embedding results for associative algebras}
\thanks{%
This preprint is readable online at
\url{http://math.berkeley.edu/~gbergman/papers/unpub}\,.
}


\author{George M. Bergman}
\address{University of California\\
Berkeley, CA 94720-3840, USA}
\email{gbergman@math.berkeley.edu}

\begin{abstract}
Suppose we wish to embed an (associative) $\!k\!$-algebra $A$
in a $\!k\!$-algebra $R$ generated in some specified way;
e.g., by two elements, or by copies of given $\!k\!$-algebras
$A_1,\ A_2,\ A_3.$
Several authors have obtained
sufficient conditions for such embeddings to exist.
We prove here some further results on this theme.
In particular, we merge the ideas of existing constructions
based on two generating {\em elements}, and on three given
{\em subalgebras,} to get a construction using two given subalgebras.

We pose some questions on how these results can be further strengthened.
\end{abstract}
\maketitle

I have decided not to publish this note -- the results are mostly
minor improvements on results in the literature; moreover, the
literature is large, and I don't have time to investigate it properly.
However, I hope that some of the ideas presented below will prove
useful for others.

Below, rings and algebras will be associative
and, except where the contrary is stated,
unital, with homomorphisms respecting $1.$
``Countable'' will mean finite or countably infinite.

I use, in several places below, techniques based on the Diamond
Lemma, in particular, on Theorems~1.2 and~6.1 of \cite{<>}.
I have worded the arguments where these are first used
so that the reader unfamiliar with \cite{<>} can see more or less
what is involved.
For precise formulations and proofs, see that paper.

For other sorts of results on embedding general $\!k\!$-algebras in
finitely generated ones, sometimes called ``affinization'',
see \cite{alg_algs} (where the emphasis is on controlling
the Gel{\mz}fand-Kirillov dimension), and works referenced there.

We remark that results of this
sort for rings were preceded in the literature, and perhaps
originally inspired by, similar results for groups.
Cf.~\cite{PES} and references given there.

\section{Algebras with few generators}\label{S.easy}

Let me begin with a result which we shall subsequently
strengthen in several ways, but which gives a simple
illustration of a technique we shall frequently use.

\begin{proposition}\label{P.easy}
Let $k$ be a commutative ring, and $A$ a countably generated
$\!k\!$-algebra which is free as a $\!k\!$-module on a basis
containing~$1.$
Then $A$ can be embedded in a $\!k\!$-algebra $R$ generated by three
elements.

In fact, given any countable generating set
$S=\{s_0,s_1,\dots,s_n,\dots\,\}$ for $A$ as a
$\!k\!$-algebra, one can take $R\supseteq A$
to have generators $x,\,y,\,z$ such that
\begin{equation}\label{d.xy^nz}\mbox{
$x\,y^n\,z\,=\,s_n$\quad$(n=0,1,\dots).$
}\end{equation}
\end{proposition}

\begin{proof}
Let $\{1\}\cup B$ be a basis for $A$ as a $\!k\!$-module,
and assume for convenience
that $B$ does not contain any of the symbols $x,\,y,\,z.$
In describing a presentation of $R,$ we will want to
distinguish between algebra elements and {\em expressions}
for those elements; so for every $a\in A,$
let $\varepsilon(a)$ denote the unique
expression for $a$ as a $\!k\!$-linear
combination of elements of $\{1\}\cup B.$

We shall prove our result by applying the Diamond Lemma,
Theorem~1.2 of \cite{<>}, to a presentation of $R,$ not in
terms of $x,\,y$ and $z,$ but in terms of the larger generating set
\begin{equation}\label{d.xyzB}\mbox{
$\{x,\,y,\,z\}\,\cup\,B,$
}\end{equation}
using both the relations which describe how
members of $B$ are multiplied in $A,$ namely
\begin{equation}\label{d.bb'}\mbox{
$b\,b'\,=\,\varepsilon(b\,b')$\quad$(b,b'\in B).$
}\end{equation}
and relations corresponding to~\eqref{d.xy^nz},
\begin{equation}\label{d.xy^nz,e}\mbox{
$x\,y^n\,z\,=\,\varepsilon(s_n)$\quad$(n=0,1,\dots),$
}\end{equation}

We view each of the relations in~\eqref{d.bb'} and~\eqref{d.xy^nz,e}
as a {\em reduction rule}, which specifies that the monomial of
length $\geq 2$ on its left-hand side is to
be reduced to the $\!k\!$-linear combination
of elements of $\{1\}\cup B$ on the right.
Note that each of these rules carries monomials to linear combinations
of shorter monomials.
Hence the partial order on the free monoid on our generating
set~\eqref{d.xyzB} that makes $s\leq t$
if and only if either $s=t,$ or $s$ is strictly shorter than $t,$
satisfies the hypotheses of \cite[Theorem~1.2]{<>}, namely,
that partial order is respected by the monoid structure, has descending
chain condition, and has the property
that the output of each of our reductions is a linear combination
of monomials $\leq$ the input monomial.
None of the monomials on the left-hand sides of our relations
are subwords of others, and the only monomials that can be formed
by overlap of two such monomials are those of the form
$b\,b'\,b''$ $(b,b',b''\in B),$ so these give the only
``ambiguities'' in the sense of \cite[\S1]{<>}.
Those ambiguities are ``resolvable'' -- i.e., the two possible
reductions that can be applied to an ambiguously reducible
monomial $b\,b'\,b'',$ when followed by appropriate further reductions,
lead to a common value --  because $A$ is associative.

Hence, by \cite[Theorem~1.2]{<>}, the algebra $R$
presented by generators~\eqref{d.xyzB}
and relations~\eqref{d.bb'} and \eqref{d.xy^nz,e}
has as a $\!k\!$-module basis the set of all monomials $w$ in
the generating set~\eqref{d.xyzB} such that
no subword of $w$ is the left-hand side
of any of the relations of~\eqref{d.bb'} or~\eqref{d.xy^nz,e}.
In particular, $\{1\}\cup B$ is a subset of this
basis, and by the relations~\eqref{d.bb'}, the
$\!k\!$-submodule of $R$ spanned by that subset of the basis is
a $\!k\!$-subalgebra isomorphic to~$A.$

Finally, by~\eqref{d.xy^nz,e}, the $\!k\!$-subalgebra generated
by $x,\,y$ and $z$ contains all the $s_n,$ hence contains
our image of $A,$ hence contains $B;$ so the three
elements $x,\,y$ and $z$ in fact generate~$R.$
\end{proof}

It was shown in \cite{LB+2} that one can in fact do the same using
two, rather than three generators, with the help of a slightly
less obvious family of monomials.

\begin{proposition}[{after \cite[Theorem~3.2]{LB+2}}]\label{P.easy2}
Let $k,$ $A$ and $S$ be as in Proposition~\ref{P.easy}.
Then $A$ can be embedded in a $\!k\!$-algebra $R$ generated by two
elements, $x$ and $y,$ so that
\begin{equation}\label{d.x^2y^nxy}\mbox{
$x^2\,y^{n+1}\,x\,y\,=\,s_n$\quad$(n=0,1,\dots).$
}\end{equation}
\end{proposition}

\begin{proof}
Note that the words on the left-hand side
of~\eqref{d.x^2y^nxy} involve $x^2$ only in the leftmost
position, and have no nonempty subword of $x^2$ at their
right end.
This limits possible overlaps or inclusions among such
words to the inclusion of one as a left segment of another.
But any two distinct words of that sort differ in the position of
the next $x$ after the initial $x^2,$ making such inclusions impossible.

The rest of the proof follows exactly the
proof of Proposition~\ref{P.easy}, with $x\,y^n\,z$ everywhere
replaced by $x^2\,y^{n+1}\,x\,y.$
\end{proof}

Alternatively, one can get the fact that a $\!k\!$-algebra
containing $A$ can be generated by two elements
from the fact that it can be generated by three, using the
lemma on p.1096 of~\cite{O+V+W}, which notes that if an algebra
(they say ``ring'', but the argument works equally for algebras)
$R$ is generated by $n$ elements $r_1,\dots,r_n,$ then the
$n{+}2\times n{+}2$ matrix ring $M_{n+2}(R),$ which contains
a copy of $R,$ can be generated by $2$ elements.
(Namely, by one matrix which permutes the $n$ coordinates
cyclically, and one having first two rows $(0,r_1,\dots,r_n,0)$
and $(1,0,\dots,0),$ and all other rows zero.)
That lemma is used in~\cite{O+V+W} to show, as in
Proposition~\ref{P.easy2} above and Theorem~\ref{T.easy}
below, that countably generated rings can be embedded in
$\!2\!$-generator rings.
But the technique for going from countable to finite generation
is quite different from that used in most of this note;
we will look at it in the final section.

The proof of Proposition~\ref{P.easy2} given
above is essentially the one given in~\cite{LB+2}, using
a tool equivalent to the Diamond Lemma, which the authors
call the method of Gr\"{o}bner-Shirshov bases and apply
with $k$ assumed a field.
(Cf.\ also \cite{LB}, \cite{LB+PK}.)
As a statement that any countably generated algebra is embeddable
in a $\!2\!$-generated algebra, the above result
is attributed there to Mal{\mz}cev~\cite{AM}.
However, we shall see in \S\ref{S.extend} that the proof
in~\cite{AM} uses a different construction, which yields
embeddings of nonunital, but not in general of unital algebras,
and which leads to some further interesting ideas.

In the two preceding results,
we applied the Diamond Lemma, but not to the generating
set $\{x,\,y,\,z\}$ or $\{x,\,y\}$
that we might have expected to use.
In the proof of the corollary below, we apply the above proposition,
but not over the base ring one might expect.

\begin{corollary}[{cf.~\cite[Proposition~2]{O+V+W}}]\label{C.k,A}
Let $k_0$ be a commutative ring, and $A$ a countably generated
{\em commutative} $\!k_0\!$-algebra.
Then there exists a \textup{(}generally noncommutative\textup{)}
$\!k_0\!$-algebra $R$ which is
generated as a $\!k_0\!$-algebra by two elements, and
contains $A$ in its center.
\end{corollary}

\begin{proof}
Let us apply Proposition~\ref{P.easy2}
with $A$ in the role of both the $k$ and the $A$ of that proposition,
and with the role of~$S$ played by
any countable generating set for $A$ over $k_0.$
(Since $A$ is free as an $\!A\!$-module on the basis
$\{1\}\cup\emptyset,$ the empty set plays the role of the $B$
used in the proof of Proposition~\ref{P.easy},
and hence implicit in the proof of Proposition~\ref{P.easy2}.)
Proposition~\ref{P.easy2} now
gives us a faithful $\!A\!$-algebra $R$ generated over $A$ by
elements $x$ and $y$ satisfying the relations~\eqref{d.x^2y^nxy}.
Since by hypothesis the elements on the right hand sides
of these relations generate $A$ over $k_0,$
the $\!k_0\!$-subalgebra of $R$ generated by $\{x,\,y\}$
contains $A,$ hence is all of $R.$
Since $R$ was constructed as an $\!A\!$-algebra, $A$ is central in $R.$
\end{proof}

For instance, taking $k=\Z,$ and for $A$ any countable commutative
ring, we get a $\!2\!$-generated $\!\Z\!$-algebra which is a
faithful $\!A\!$-algebra.
Here $A$ might be $\Q,$ or an extension field of
$\Q$ of countable transcendence degree; or it might be a commutative
ring having any countable Boolean ring as its Boolean
ring of idempotents.
(I gave a construction of a finitely generated algebra with
an infinite Boolean ring of
central idempotents, using~\eqref{d.xy^nz}, at the
end of~\S12.2 of~\cite{prod_Lie1}.
The present note had its origin in thinking about how that
construction might be generalized.)

The contrast with commutative algebras is striking.
If $k$ is a field and $R$ a finitely generated commutative
$\!k\!$-algebra, I claim that any subfield $A$ of $R$ containing $k$
must be a finite extension of $k$ (i.e., finite-dimensional).
For by Theorem~IX.1.1 of \cite{SL.Alg}, $R$ admits a
homomorphism $h$ into the algebraic closure $\bar{k}$ of $k;$ so
as $R$ is finitely generated as an algebra and $\bar{k}$ is algebraic,
$h(R)$ must be a finite extension field of $k.$
Since $A$ is a field, $h$ is one-to-one
on $A,$ so $A$ itself must be finite over $k,$ as claimed.
Likewise, since a finitely generated commutative algebra over
a field is Noetherian, its Boolean ring
of idempotents cannot be infinite.

It is noted in~\cite[proof of Corollary~2]{O+V+W} that any
finitely generated $\!\Z\!$-algebra which contains $\Q$
is an example of a $\!\Q\!$-algebra which {\em cannot}
be written $R\otimes_{\Z}\Q$ for $R$ a $\!\Z\!$-algebra
which is free as a $\!\Z\!$-module.

\section{More general module-structures}\label{S.bimod}

Propositions~\ref{P.easy} and~\ref{P.easy2}, which we proved using
the ``everyday'' version of the Diamond Lemma, require
the algebra $A$ to be free as a $\!k\!$-module on a basis
of the form $B\cup\{1\}.$
Using the bimodule version of the Diamond Lemma,
we can drop that condition.

\begin{theorem}[cf.~{\cite[Theorem on p.1097]{O+V+W}}]\label{T.easy}
Let $k$ be a commutative ring, and $A$ any $\!k\!$-algebra generated
as a $\!k\!$-algebra by a countable set $\{s_0,\,s_1,\,\dots\,\}.$
Then $A$ can be embedded in a $\!k\!$-algebra $R$ generated by three
elements $x,\,y,\,z$ so that\textup{~\eqref{d.xy^nz}} holds,
and also in a $\!k\!$-algebra generated by two
elements $x,\,y$ so that~\eqref{d.x^2y^nxy} holds.
\end{theorem}

\begin{proof}
We shall prove the case based on~\eqref{d.xy^nz}.
The case based on~\eqref{d.x^2y^nxy} is exactly analogous.

Given $A$ and $\{s_i\}$ as above, let us
use the bimodule version of the Diamond Lemma,
Theorem~6.1 of \cite{<>}, with our $A$ in the role of the
(not necessarily commutative) ring called $k$ in that theorem.
We begin by taking three $\!(A,A)\!$-bimodules freely generated by
$\!k\!$-centralizing elements $x,$ $y$ and $z,$
\begin{equation}\label{d.AxA++}\mbox{
$M_x\,=\,A\,x\,A\,\cong\,A\otimes_k A,$\qquad
$M_y\,=\,A\,y\,A\,\cong\,A\otimes_k A,$\qquad
$M_z\,=\,A\,z\,A\,\cong\,A\otimes_k A,$
}\end{equation}
and forming the tensor ring $A\langl M_x\oplus M_y\oplus M_z\rangl$
on their direct sum.
If we grade this ring in the obvious way by the free monoid
on $\{x,\,y,\,z\},$ its homogeneous component indexed by
each word $x\,y^n\,z$ $(n\geq 0)$ is the product
\begin{equation}\begin{minipage}[c]{35pc}\label{d.Ax...zA}
$M_x\,(M_y)^n\,M_z\ =\ A\,x\,(A\,y)^n\,A\,z\,A\\[.2em]
{\vrule width3em height0pt depth0pt} %
\ \cong(A\,x\,A)\otimes_A (A\,y\,A)\otimes_A \dots
\otimes_A (A\,y\,A)\otimes_A (A\,z\,A)\ \cong
\ A\otimes_k A\otimes_k\dots \otimes_k A\quad (n+3\ A\mbox{'s}).$
\end{minipage}\end{equation}
We now impose on $A\langl M_x\oplus M_y\oplus M_z\rangl$ relations
determined by reduction maps sending the homogeneous component
indexed by each word $x\,y^n\,z$ to the component $A$ (indexed by
the empty word $1)$ using the $\!(A,A)\!$-bimodule homomorphism
that acts on reducible elements of~\eqref{d.Ax...zA} by
\begin{equation}\label{d.a_0...n+2}\mbox{
$a_0\,x\,a_1\,y\,a_2\,\dots\,a_n\,y\,a_{n+1}\,z\,a_{n+2}
\ \longmapsto\ a_0\,a_1\,a_2\,\dots\,a_n\,a_{n+1}\,s_n\,\,a_{n+2}.$
}\end{equation}
To see that such a homomorphism exists, we note first that the
description of the bimodule~\eqref{d.Ax...zA} as an
$\!n{+}3\!$-fold tensor product over $k$ of copies of $A$
(last step of~\eqref{d.Ax...zA})
shows that~\eqref{d.a_0...n+2} determines a
{\em $\!k\!$-module} homomorphism, by the universal property
of $\otimes_k.$
Looking at how the right-hand side
of~\eqref{d.a_0...n+2} depends on $a_0$ and $a_{n+2},$ we
see that this map is in fact a homomorphism of $\!(A,A)\!$-bimodules.

We note next that the case of~\eqref{d.a_0...n+2}
where $a_0=\dots=a_{n+2}=1$ shows that this
map carries the left hand side
of~\eqref{d.xy^nz} to the right hand thereof.

Because the family of reductions $x\,y^n\,z\mapsto 1$
in the free monoid on $\{x,\,y,\,z\},$ which indexes the components
of $A\langl M_x\oplus M_y\oplus M_z\rangl,$ has no ambiguities,
Theorem~6.1 of \cite{<>} shows that the $\!A\!$-ring $R$ presented by
the $\!(A,A)\!$-bimodules $M_x,$ $M_y,$ $M_z$ and the relations equating
inputs and outputs of each bimodule homomorphism~\eqref{d.a_0...n+2}
is the direct sum of all iterated tensor products of those three
bimodules in which no subproduct
$M_x\,(M_y)^n\,M_z$ $(n\geq 0)$ occurs.
In particular, the component of $R$ indexed by the monoid element $1$
is one of these summands, and is a copy of the algebra~$A.$

A priori, the ring $R$ we have constructed is generated by $A$
and the three bimodules~\eqref{d.AxA++}.
But since the relations~\eqref{d.xy^nz} and the structure of
$A$ allow us to express all elements
of $A$ in terms of $x,$ $y,$ $z$ and the elements of $k,$
and since all elements of $M_x=A\,x\,A,$ $M_y=A\,y\,A$ and $M_z=A\,z\,A$
can then be expressed using these elements and, again,
$x,$ $y$ and $z,$ we see that $R$ is, as claimed,
generated over $k$ by those three elements, completing
the proof of the case of the theorem based on generators
$x,$ $y$ and $z$ and relations~\eqref{d.xy^nz}.

The case based on generators $x$ and $y$ and
relations~\eqref{d.x^2y^nxy} is exactly analogous.
\end{proof}

Incidentally, we could have carried out the above
construction equally well with the terms on the right-hand side
of~\eqref{d.a_0...n+2}
permuted in any way which left $a_0$ and $a_{n+2}$ fixed.
Those two have to be placed as shown, to make the
map an $\!(A,A)\!$-bimodule homomorphism, but the decision on how
to order the others, in particular, of where to place the $s_n,$
was quite arbitrary.

\section{The ideal extension property}\label{S.extend}

Let us look at Theorem~\ref{T.easy} from a different point of view.

\begin{definition}\label{D.extend}
If $A\subseteq B$ are algebras, we will say that $A$ has
the {\em ideal extension property} in $B$ if every ideal $I\subseteq A$
is the intersection of $A$ with an ideal $J\subseteq B;$ equivalently,
is the intersection of $A$ with the ideal $B\,I\,B$
of $B$ that it generates.
\end{definition}

(This is somewhat like the {\em lying-over}
property of commutative ring theory; but since the latter
concerns prime ideals, we do not use that name, but one
modeled on the {\em congruence extension} property
of universal algebra \cite[p.412]{GG}.)

Now in Theorem~\ref{T.easy}, the countably generated algebra
$A$ can be thought of as the factor-algebra of
the free $\!k\!$-algebra on a countably infinite set of generators
by an {\em arbitrary} ideal $I;$ and the conclusion shows us that
$k\langl x,\,y,\,z\rangl$ (respectively, $k\langl x,\,y\rangl)$
has a homomorphic image in which the free subalgebra generated
by the elements $x\,y^n\,z$ (respectively, $x^2\,y^{n+1}\,x\,y)$
collapses to an isomorphic copy of $A.$
So the theorem says that %
that free subalgebra on countably many generators has the
ideal extension property within the given $\!2\!$- or
$\!3\!$-generator free algebra.

In fact, the method of proof of that theorem
clearly shows the following.

\begin{corollary}[to proof of Theorem~\ref{T.easy}]\label{C.extend}
Let $k$ be a commutative ring, $X$ a set, and $W$
a family of nonempty words in the elements of $X$
\textup{(}i.e., elements of the free semigroup on $X),$
such that no member of $W$ is a subword of another,
and no nonempty proper final subword of a member of $W$ is
also an initial subword of a member of $W.$
Then the subalgebra of $k\langl X\rangl$ generated
by $W$ is free on $W,$ and the inclusion
$k\langl W\rangl\subseteq k\langl X\rangl$
has the ideal extension property.
\end{corollary}

\begin{proof}
Let $I$ be an ideal of $k\langl W\rangl,$
let $A = k\langl W\rangl/I,$ and imitate
the proof of Theorem~\ref{T.easy}.
(Note that the ``no inclusions and no overlap'' assumption on $W$
is precisely what is needed for a
reduction system mapping into $A$ every tensor product
$(A x_1 A)\otimes_k\dots\otimes_k (A x_n A)$
such that $x_1\dots x_n\in W$ to have no ambiguities.)
\end{proof}

\section{Some old results on nonunital embeddings}\label{S.nonunital}

The two earliest results I am aware of which showed that
wide classes of associative algebras could be embedded in two-generator
algebras, Theorem~3 of \cite{AM} and Lemma~2 of \cite{AS},
were obtained by methods that, in effect,
established the ideal extension property
(Definition~\ref{D.extend} above) for certain free subalgebras
of free algebras without the use of anything like the Diamond Lemma
(and, consequently, did not yield normal forms for the
$\!2\!$-generator algebras $R$ obtained).

Those results concerned nonunital algebras, so we make

\begin{definition}\label{D.nonunital}
In this section, $\!k\!$-algebras,
though still associative, will not be assumed unital.
\textup{(}In formal statements we will make this explicit,
using the word ``nonunital'', meaning ``not necessarily unital''.
On the other hand, our commutative base ring $k$ will
continue to be unital.\textup{)}

The free nonunital $\!k\!$-algebra on a set $X$ will
be denoted $[k]\langl X\rangl.$
For $R$ a nonunital $\!k\!$-algebra, we shall
write $k+R$ for the unital $\!k\!$-algebra obtained by
universally adjoining a unit to $R.$
Thus, $k+R$ has underlying $\!k\!$-module $k\oplus R.$

The {\em ideal extension property} for nonunital algebras will
be defined as for unital algebras.
The one formal change required is that the
ideal $J$ of $B$ generated by $I\subseteq A$
must be described as $(k+B)\,I\,(k+B)$ rather than $B\,I\,B.$
\end{definition}

Let us give a name to a property which is implicit in the arguments of
\cite{AM} and \cite{AS}.

\begin{definition}\label{D.isolated}
We shall call a subsemigroup $S$ of a semigroup $T$ {\em isolated}
if for all $t,\,t'\in T\cup\{1\}$ and $s\in S,$
one has $t\,s\,t'\in S\implies t,\,t'\in S\cup\{1\}.$
\textup{(}Here we write $\cup\{1\}$ for the construction of
adjoining $1$ to a semigroup, to get a monoid.\textup{)}
\end{definition}

\begin{lemma}[{after \cite[proof of Lemma~2]{AS}}]\label{L.iso>ext}
If $S$ is an isolated subsemigroup of a
semigroup $T,$ and $k$ is any commutative
ring, then the semigroup algebra $k\,S$ has the ideal extension
property in the semigroup algebra~$k\,T.$
\end{lemma}

\begin{proof}
If $I$ is an ideal of $k\,S,$ then the general element
of the ideal $J=(k+k\,T)\,I\,(k+k\,T)$ generated by $I$
in $k\,T$ can be written
\begin{equation}\label{d.tit}\mbox{
$g\ =\ \sum_{i=1}^n t_i\,f_i\,t'_i,$
}\end{equation}
where each $f_i$ lies in $I,$ and all $t_i$ and $t'_i$
lie in $T\cup\{1\}.$
Let us write $g=g'+g'',$ where $g'$ is the sum of
those terms of~\eqref{d.tit} which have both
$t_i$ and $t'_i$ in $S\cup\{1\},$ and $g''$ is the sum of
all other terms.
Then clearly $g'\in I,$ while by the assumption that $S$ is isolated
in $T,$  the element $g''$ is a $\!k\!$-linear combination
of elements of $T-S.$
Hence if $g\in k\,S,$
we must have $g''=0,$ so $g=g'\in I.$
This shows that $J\cap k\,S=I,$ as required.
\end{proof}

We also note

\begin{lemma}\label{L.iso>free}
In a free semigroup, every isolated subsemigroup is free.
\end{lemma}

\begin{proof}
Let $S$ be an isolated
subsemigroup of the free semigroup $T,$ and $W$ the set
of elements of $S$ that cannot be factored within $S.$
Then every member of $S$ can be written as a
product of members of $W,$ and it suffices to show
that this factorization is unique.
Suppose
\begin{equation}\label{d.us&vs}\mbox{
$u_1\dots u_m\ =\ v_1\dots v_n\quad(m,\,n\geq 2,\ u_i,\,v_j\in W),$
}\end{equation}
and assume inductively that for every member of $S$ of smaller
length in the free generators of $T,$ the expression
as a product of members of $W$ is unique.
Without loss of generality, we may assume the length
of $v_1$ in the free generators of $T$ to be greater than
or equal to that of $u_1,$
and so write $v_1=u_1\,w$ for some $w\in T\cup\{1\}.$
Applying the definition of isolated subsemigroup
to the equation $v_1=1\cdot u_1\cdot w,$ we conclude
that $w\in S\cup\{1\}.$
Hence as $v_1$ cannot be factored in $S,$ we must have
$w=1,$ hence $v_1=u_1;$
so~\eqref{d.us&vs} implies $u_2\dots u_m\ =\ v_2\dots v_n.$
By our inductive assumption, these factorizations are the
same; so the two factorizations of~\eqref{d.us&vs} are the same.
\end{proof}

Remark: If $S\subseteq T$ are {\em monoids}, then $S$ is
isolated in $T$ if and only if it is closed under taking factors.
(``If'' is immediate;
``only if'' can be seen by applying Definition~\ref{D.isolated}
with $s=1.)$
Hence the isolated submonoids of a free monoid are just the
submonoids generated by subsets of the free generating
set, which will be uninteresting for our purposes.
But there are many interesting isolated subsemigroups of
free semigroups.
The next result notes a family of examples
implicit in the two papers referred to.

\begin{lemma}\label{iso>extn}
If $f$ is a function from the positive integers to the
positive integers,
then in the free semigroup $T$ on two generators $x$ and $y,$
the subsemigroup $S$ generated by all elements
\begin{equation}\label{d.xy^nx^fn}\mbox{
$x\,y^n\,x^{f(n)}\quad (n\geq 1)$
}\end{equation}
is isolated.
\textup{(}The case where $f(n)=n$ is used by Mal{\mz}cev~\cite{AM};
the case $f(n)=1$ by Shirshov~\cite{AS}.\textup{)}
\end{lemma}

\begin{proof}
It is not hard to see that given a product $u$ of elements of
the form~\eqref{d.xy^nx^fn}, the factors in question begin precisely
at the points in $u$ where a sequence $xy$ occurs.
Hence, marking a break before each such point, we can recover
the factorization into such elements.
(So in particular, the semigroup $S$ is free on the set of
elements~\eqref{d.xy^nx^fn}.)

Now if such an product $u\in S$ has a factorization
$u=t\,v\,t'$ with $t,\,t'\in T\cup\{1\}$ and $v$
of the form~\eqref{d.xy^nx^fn},
then one of our break points occurs at the beginning of $v;$
hence the factor $v$ begins at the same point of $u$ as one of
the factors in our expression for $u$ as a product of
elements~\eqref{d.xy^nx^fn}.
But it is easy to check that no element~\eqref{d.xy^nx^fn} is a
proper left divisor of any other; so $v$ is in fact a term of our
factorization of $u$ into elements~\eqref{d.xy^nx^fn}.
From this it follows that, more generally, if $u\in S$
has a factorization $t\,s\,t'$ with $s$ a {\em product}
of elements~\eqref{d.xy^nx^fn}, that is, a member of $S,$
then $s$ is a substring of our expression for $u$ as such
a product, hence each of $t$ and $t'$ is either such a
substring or empty, proving that $S$ is indeed isolated in $T.$
\end{proof}

Combining the last three lemmas, we have

\begin{theorem}[{after Mal{\mz}cev \cite[Theorem~3]{AM}, Shirshov \cite[Lemma~2]{AS}}]\label{T.M&A}
Let $k$ be a commutative ring, let $B$ be the free nonunital
associative algebra $[k]\langl x,\,y\rangl,$ and let
$A$ be either the subalgebra of $B$ generated by all monomials
$x\,y^n\,x,$ or the subalgebra generated by all monomials $x\,y^n\,x^n$
\textup{(}or, more generally,
the subalgebra generated by all monomials $x\,y^n\,x^{f(n)}$
for any function $f$ from the positive integers to the
positive integers\textup{)} for $n\geq 1.$

Then $B$ is a free algebra on the indicated countably infinite
generating set, and has the ideal extension property in $A.$

This gives, for {\em nonunital} algebras, another way of embedding
an arbitrary countably generated
$\!k\!$-algebra in a $\!2\!$-generator $\!k\!$-algebra.\qed
\end{theorem}

We remark that Shirshov's statement of \cite[Lemma~2]{AS}
leaves it unclear whether unital or nonunital algebras are intended.
However, in the unital case, if we write $B=k\langl x,\,y\rangl,$ the
unital subalgebra $A\subseteq B$ generated by the elements $x\,y^n\,x$
does {\em not} have the ideal extension property, which his
proof would require.
For example,
let $I$ be the ideal of $A$ generated by $x\,y\,x$ and $x\,y^2 x-1.$
Clearly $I$ is proper, since the factor-algebra $A/I$ is free on the
free generators $x\,y^n\,x$ of $A$ other than $x\,y\,x$ and $x\,y^2 x.$
However, the ideal $J$ that it generates in $B$ is improper, since
in $B/J,$ the element $x\,y\,x\,y^2\,x\,y\,x$ reduces, on the one hand,
to $0,$ in view of the factors $x\,y\,x,$ while on the other
hand, if we simplify
the middle factor $x\,y^2\,x$ to $1,$ and then do the same
to the resulting monomial, we get $1,$ so $0=1$ in $B/J.$
Thus, the algebras of \cite[Lemma~2]{AS} should be
understood to be nonunital.

The converse of Lemma~\ref{L.iso>free} above is not true.
For example, in the free semigroup $T$ on one generator $x,$
the subsemigroup $S$ generated by $x^2$ is not isolated
(since $x\cdot x^2\cdot x\in S),$ but $k\,S\subseteq k\,T$
does have the ideal extension property for every $k.$
This suggests

\begin{question}\label{Q.extend}
Is there a nice characterization of
the inclusions $S\subseteq T$ of semigroups \textup{(}respectively,
monoids\textup{)} for which the inclusion of nonunital
\textup{(}respectively, unital\textup{)}
$\!k\!$-algebras $k\,S\subseteq k\,T$
has the ideal extension property?

In particular, what can one say in the cases where $S$
and $T$ are free as semigroups or monoids?

\textup{(}One expects the answers to the above questions
to be independent of $k,$ but there is no evident reason why
this must be true.
It is not too implausible that it might
depend on the characteristic of $k.)$
\end{question}

\section{Embedding in algebras generated by a given family of algebras}\label{S.theta}

The condition of countable generation on the algebra
$A$ in the results of the preceding sections cannot be dropped.
For instance, if $k$ is a field, every finitely
generated $\!k\!$-algebra is countable-dimensional, hence so is every
algebra embeddable in such an algebra.
So, for example, a finitely
generated algebra over the field $\R$ of real numbers cannot
contain a copy of the rational function
field $\R(t),$ since that is continuum-dimensional.

To get around this difficulty, we might vary the construction
of Proposition~\ref{P.easy} by considering
$\!k\!$-algebras generated by elements $x$ and $z$ together with
all formal real powers $y^r$ $(r\in\R)$ of the symbol $y.$
We would then have enough expressions $x\,y^r\,z$ to hope
to get any continuum-generated $\!k\!$-algebra $A.$
In effect, we would be looking at $\!k\!$-algebras generated
by $x,$ $z$ and a copy of the group algebra $k\,G,$ where
$G$ is the additive group of the real numbers, written
multiplicatively as formal powers of $y.$

We can, in fact, get such results with $k\,G$ replaced by
a fairly general $\!k\!$-algebra.
Here is one such statement (where ``$A_0$'' is the algebra we
want to embed, and ``$A_1$'' the algebra generalizing $k\,G).$

\begin{theorem}\label{T.A0A1}
Suppose $A_0$ and $A_1$ are faithful algebras over a commutative
ring $k,$ such that $k$ is a module-theoretic direct summand
in each, and such that $A_0$ is generated as a $\!k\!$-algebra
by the image of a $\!k\!$-{\em module}
homomorphism $\varphi: A_1\to A_0.$

Then $A_0$ can be embedded in a $\!k\!$-algebra $R$ generated
over $A_1$ by two elements $x$ and $z$ satisfying
\begin{equation}\label{d.varphi}\mbox{
$x\,a\,z\ =\ \varphi(a)\quad (a\in A_1).$
}\end{equation}
\textup{(}Here for notational convenience we are identifying
$A_0$ and $A_1$ with their embedded images in $R.)$
\end{theorem}

\begin{proof}
Our first step will be to embed $A_0$ and $A_1$ in a common
$\!k\!$-algebra $A$ having a $\!k\!$-module endomorphism $\theta$ that
carries $A_1$ to our generating subset of $A_0.$
To do this, let us choose $\!k\!$-module decompositions of the
sort whose existence is assumed in the hypothesis,
\begin{equation}\label{d.A01=k+}\mbox{
$A_0=k\oplus M_0,\quad A_1=k\oplus M_1.$
}\end{equation}
Letting
\begin{equation}\label{d.A0(X)A1}\mbox{
$A=A_0\otimes_k A_1$
}\end{equation}
(made a $\!k\!$-algebra in the usual way), we see from~\eqref{d.A01=k+}
that the $\!k\!$-algebra homomorphisms
of $A_0$ and $A_1$ into $A$ given by
$a_0\mapsto a_0\otimes 1$ and $a_1\mapsto 1\otimes a_1$ are embeddings.
Letting $\pi: A_0\to k$ be the $\!k\!$-module projection along
$M_0,$ we find that the $\!k\!$-module endomorphism
$\theta$ of $A$ given by
\begin{equation}\label{d.pipsi}\mbox{
$\theta(a_0\otimes a_1)\,=\,\varphi(a_1)\otimes\pi(a_0)$
}\end{equation}
carries $k\otimes_k A_1,$ our copy of $A_1,$
onto $\varphi(A_1)\otimes_k k,$ the generating $\!k\!$-submodule for
our copy of $A_0.$

Now that we have $A$ and $\theta,$ the remainder of our proof
is like that of Theorem~\ref{T.easy}, but simpler.
We take two (rather than three) free $\!k\!$-centralizing
$\!(A,A)\!$-bimodules,
\begin{equation}\label{d.AxA+}\mbox{
$M_x\,=\,A\,x\,A\,\cong\,A\otimes_k A,$\qquad
$M_z\,=\,A\,z\,A\,\cong\,A\otimes_k A,$
}\end{equation}
form the tensor ring $A\langl M_x\oplus M_z\rangl$
on their direct sum, and impose the relations determined by
a single bimodule homomorphism from the component indexed by $x\,z,$
namely $M_x\otimes_A M_z=A\,x\,A\,z\,A,$ to the component
indexed by $1,$ namely $A,$ where this
homomorphism is defined to act on generators by
\begin{equation}\label{d.aa'a''}\mbox{
$a\,x\,a'\,z\,a''\ \longmapsto\ a\,\theta(a')\,a''.$
}\end{equation}
On the indexing free monoid on $\{x,\,z\},$ this corresponds
to the single reduction $x\,z\mapsto 1,$ which has no ambiguities.
As in the proof of Theorem~\ref{T.easy},
we deduce that the relations corresponding
to~\eqref{d.aa'a''} define a $\!k\!$-algebra $R$ in which $A$
is embedded.
Hence $A_0$ and $A_1$ are embedded in $R,$ where they
satisfy~\eqref{d.varphi}.
But that relation shows that the subalgebra of $A$
generated by $x,\,z$ and $A_1$ contains $A_0;$
so that subalgebra is all of $A,$ as required.
\end{proof}

In fact, there is a result in the literature which achieves
much greater generality in some ways
(though in others it is more restricted).
Bokut{\mz} shows in Theorems~1 and~$1'$ of \cite{LB}
that for any four nonzero nonunital algebras $A_0,$ $A_1,$ $A_2,$ $A_3$
over a field $k,$ one can embed $A_0$ in an algebra $R$
generated by the union of one copy of each of $A_1,$ $A_2$ and $A_3,$
as long as $A_0$ satisfies the obvious restriction of
having $\!k\!$-dimension less than or equal to that
of the $\!k\!$-algebra coproduct of $A_1,$ $A_2$ and $A_3$
(namely, $\max(\aleph_0,\,\dim A_1,\,\dim A_2,\,\dim A_3)),$
and (for less obvious reasons; but see note at
reference \cite{LB+2} below)
as long as $\r{card}\,k$ is less than or equal that same dimension.
Moreover, Bokut{\mz}'s construction makes $R$ a simple $\!k\!$-algebra!

So Theorem~\ref{T.A0A1}, in the case where
$k$ is a field, and our algebras are nonunital, and the
cardinality of $k$ satisfies the indicated bound,
is majorized by the particular case of Bokut{\mz}s result
where $A_2$ and $A_3$ are free algebras on single
generators $x$ and $z.$

Given that Proposition~\ref{P.easy2} and Theorem~\ref{T.M&A} above
improve on our original $x\,y^n\,z$ construction by using
two generators rather than three, it is natural to ask whether
one can get a result that embeds an algebra $A_0$
in an algebra $R$ generated by copies of two given algebras, $A_1$
and $A_2,$ rather than the three of the result quoted.
We obtain such a result, Theorem~\ref{T.A0A1A2}, below
(though the algebras allowed are not quite as general as I would like;
and I do not attempt to make $R$ simple).

Let us recall, before going further, that nonunital $\!k\!$-algebras
$R$ correspond to unital $\!k\!$-algebras $R'$ given
with augmentation homomorphisms $\pi: R'\to k,$ via the
constructions $R'=k+R$ and $R=\ker(\pi),$ and that these
constructions in fact give an equivalence between the category
of nonunital $\!k\!$-algebras and the category
of augmented unital $\!k\!$-algebras.
From this point of view, the condition
in Theorem~\ref{T.A0A1} above that $A_0$ and
$A_1$ each have $k$ as a $\!k\!$-module direct summand
is a weakened version of nonunitality, a ``module-theoretic
augmentation'' rather than a ring-theoretic one.
In the next result, we likewise have augmentation-like
conditions of various strengths on the three given algebras.
That is not surprising,
since the result is modeled on Theorem~\ref{T.M&A}.

\begin{theorem}\label{T.A0A1A2}
Let $k$ be a commutative ring, and let $A_1$ and $A_2$ be
$\!k\!$-algebras such that

\textup{(i)}~the structure map $k\to A_1$
admits a module-theoretic left inverse $\pi,$
whose kernel we shall denote $M_1,$ and

\textup{(ii)}~$A_2$ admits a surjective $\!k\!$-algebra homomorphism
$\psi: A_2\to k[x]/(x^3),$ whose kernel we shall denote~$M_2.$
We shall, by abuse of notation, use the same symbol $x$
for the image in $k[x]/(x^3)$ of $x\in k[x],$ and also
for a fixed inverse image, in $A_2,$ of that element
of $k[x]/(x^3)$ under $\psi.$

Then in the coproduct \textup{(}``free product''\textup{)}
\begin{equation}\label{d.B=}\mbox{
$B\ =\ A_1\smallcoprod A_2$
}\end{equation}
of $A_1$ and $A_2$ as $\!k\!$-algebras,
the $\!k\!$-submodule $x\,M_1\,x$ is isomorphic to
$M_1,$ and generates a
nonunital $\!k\!$-subalgebra $A$ isomorphic to the
nonunital tensor algebra $[k]\langl M_1\rangl;$
and this subalgebra $A$ has the ideal extension property in~$B.$

Hence, any $\!k\!$-algebra $A_0$ which admits a
$\!k\!$-algebra homomorphism to $k$ \textup{(}an augmentation\textup{)},
and which can be generated as a $\!k\!$-algebra by a module-theoretic
homomorphic image of $M_1,$ can be embedded in a $\!k\!$-algebra
$R$ generated by an image of $A_1$ and an image of $A_2.$
Moreover, these images can be taken to be
isomorphic copies of those two algebras.
\end{theorem}

\begin{proof}
By~(i),
\begin{equation}\label{d.A1=}\mbox{
$A_1\ =\ k\oplus M_1$
}\end{equation}
as $\!k\!$-modules, while~(ii) leads to a decomposition
\begin{equation}\label{d.A2=&}\mbox{
$A_2\ =\ k\oplus k\,x\oplus k\,x^2\oplus M_2.$
}\end{equation}
So writing
\begin{equation}\label{d.M'2}\mbox{
$M'_2\ =\ k\,x\oplus k\,x^2\oplus M_2,$
}\end{equation}
we have
\begin{equation}\label{d.A2=}\mbox{
$A_2\ =\ k\oplus M'_2.$
}\end{equation}

By Corollary~8.1 of \cite{<>}, the decompositions~\eqref{d.A1=}
and~\eqref{d.A2=} lead to a decomposition of the $\!k\!$-algebra
coproduct $B=A_1\smallcoprod A_2$
as the $\!k\!$-module direct sum of all alternating tensor products
\begin{equation}\label{d.M1M2M1M2}\mbox{
$\dots\otimes_k M_1\otimes_k M'_2 \otimes_k M_1\otimes_k
M'_2\otimes_k\dots,$
}\end{equation}
(where each such tensor product may begin with either
$M_1$ or $M'_2$ and end with either $M_1$ or $M'_2,$ and where
we understand the unique length-$\!0\!$ tensor
product to be $k,$ and the two length-$\!1\!$
products to be $M_1$ and $M'_2).$

Using~\eqref{d.M'2}, we can now refine this decomposition,
writing $B$ as the direct sum of submodules each of which is
\begin{equation}\begin{minipage}[c]{35pc}\label{d.M1*M1*}
a tensor product such that, as in~\eqref{d.M1M2M1M2}, every
other term is $M_1,$ but where each of the remaining
terms can be any of the three
$\!k\!$-modules $k\,x,$ $k\,x^2,$ or $M_2.$
\end{minipage}\end{equation}

Let us note that if we multiply two of
the summands~\eqref{d.M1*M1*} together within
$B,$ the result
will {\em often} lie entirely within a third.
The exception is when the first factor ends with $M_1$
and the second begins with $M_1,$ in which case the relation
\begin{equation}\label{d.M1+k}\mbox{
$M_1\,M_1\ \subseteq\ k+M_1,$
}\end{equation}
arising from the relatively weak module-theoretic
hypothesis~(i) on $A_1,$ leads to two such summands.

We now consider the summand
\begin{equation}\label{d.xM1x}\mbox{
$(k\,x)\otimes_k M_1\otimes_k(k\,x)\ =\ x\,M_1\,x\ \cong\ M_1,$
}\end{equation}
of $B,$ and the nonunital subalgebra of $B$ it generates, which we name
\begin{equation}\label{d.A}\mbox{
$A\ =\ [k]\langl x\,M_1\,x\rangl.$
}\end{equation}
Clearly, when we multiply~\eqref{d.xM1x} by itself an arbitrary positive
number of times, there are no cases
of a tensor product ending in $M_1$ being multiplied by
one beginning with $M_1;$ so the product takes the form
\begin{equation}\label{d.xMx^n}\mbox{
$(k\,x)\,M_1\,(k\,x^2)\,M_1\,(k\,x^2)\,\dots\,
(k\,x^2)\,M_1\,(k\,x)\ \cong
\ M_1\otimes_k M_1\otimes_k \dots\otimes_k M_1$\qquad
(with $\geq 1$ $M_1$'s).
}\end{equation}
(To see the isomorphism, note that $k\,x^2\cong k\,x\cong k$ as
$\!k\!$-modules, and $-\otimes_k k \otimes_k -$ simplifies
to $-\otimes_k -.)$
Thus~\eqref{d.A} is, as claimed, isomorphic to the nonunital tensor
algebra on the $\!k\!$-module $M_1.$

Now suppose we multiply one of the summands~\eqref{d.xMx^n} both
on the left by a summand~\eqref{d.M1*M1*}
and on the right by a summand~\eqref{d.M1*M1*}.
Again, because of the form of~\eqref{d.xMx^n},
this does not lead to an $M_1$ being multiplied by another $M_1,$
so the product always lies in a single summand~\eqref{d.M1*M1*}.
The reader should verify that this
summand will again have the form~\eqref{d.xMx^n} if
and only if the left factor and the right factor
are each either $k$ or of the form~\eqref{d.xMx^n}.
Thus, the summands~\eqref{d.xMx^n} form something like an
isolated subsemigroup among the summands~\eqref{d.M1*M1*};
though we can't quite use that concept,
since the summands~\eqref{d.M1*M1*} don't form
a semigroup in a natural way, in view of~\eqref{d.M1+k}.

We can now reason as in the proof of
Lemma~\ref{L.iso>ext}: given an ideal $I\subseteq A,$
let $J$ be the ideal of $B$ that it generates.
The general element of $J$ can be written in the form~\eqref{d.tit},
i.e., $\sum_{i=1}^n t_i\,f_i\,t'_i,$
where each $f_i\in I,$ while each $t_i$ and each $t'_i$ lies in
a summand~\eqref{d.M1*M1*}.
Those terms of~\eqref{d.tit} where both $t_i$ and $t'_i$ lie in
summands that are either $k$ or of the form~\eqref{d.xMx^n},
and so belong to $k+A,$
will again belong to $I,$ while all other summands
will, by the result of the preceding
paragraph, have values in the $\!k\!$-submodule of
$B$ spanned by the summands~\eqref{d.M1*M1*}
not of the form~\eqref{d.xMx^n}.
Hence if~\eqref{d.tit} lies in $A,$
the sum of all terms of the latter sort must be zero.
Hence our expression~\eqref{d.tit} will equal the sum of the terms
of the first
sort, hence lie in $I,$ establishing the ideal extension property.

It is also easy to verify that for $I$ and $J$ as above,
elements of $J$ have zero components in
the summands $k,$ $M_1,$ and $k\,x+k\,x^2+M_2$ of $B.$
Thus, the images of $A_1$ and $A_2$ in $R=B/J$ are faithful.

Now from the assumption that $A_0$ can be generated
as a {\em unital} $\!k\!$-algebra by a homomorphic image of $M_1,$
it is easy to see that the kernel of its augmentation map
-- let us call that kernel $M_0$ --
can be generated as a {\em nonunital} $\!k\!$-algebra by such an image,
hence, since $A$ is isomorphic
to the nonunital tensor algebra on $M_1,$ this algebra
$M_0$ is isomorphic to $A/I$ for some ideal $I\subseteq A.$
By the preceding arguments, $A/I$ embeds in the $\!k\!$-algebra
$R=B/J,$ where $J=B\,I\,B.$
We thus have $A_0=k+M_0\cong k+A/I\subseteq B/J,$
giving the desired embedding of $A_0$ in an algebra
generated by embedded copies of $A_1$ and~$A_2.$
\end{proof}

The statement of Theorem~\ref{T.A0A1A2} is not particularly elegant.
(If we had assumed $k$ a field, we could have dropped
condition~(i), which is automatic in that case,
making the statement a little nicer.
If, instead, we had worked with nonunital rings, we could have dropped
both that and the augmentation assumption on $A_0,$
and also shortened the proof.)
Nor can Theorem~\ref{T.A0A1A2} claim to
be the strongest possible result of this sort.
E.g., with slightly different assumptions on $A_1$ and $A_2,$
we could have weakened the assumption that $A_0$ was
generated by one image of $M_1$ to allow it to
be generated by a countable family of such images.

However, the proof of the theorem, as given, illustrates
nicely several techniques that can be used in such situations.

I do not know a way of avoiding the need
for something like the assumption in that theorem that
$A_2$ admit a homomorphism onto $k[x]/(x^3),$ even if $k$ is a field.
But there are no evident examples
showing that embeddability fails without such an assumption;
so let us ask the following question.
(Note that algebras are unital, since the contrary is not stated.)

\begin{question}\label{Q.A1A2}
Suppose $k$ is a field, and $A_1,\ A_2$ are $\!k\!$-algebras,
both of which have $\!k\!$-dimension $\geq 2,$ and at least
one of which has $\!k\!$-dimension $\geq 3.$

Can every $\!k\!$-algebra $A_0$ with
$\dim_k A_0\leq \max(\aleph_0,\,\dim A_1,\,\dim A_2)$
be embedded in a $\!k\!$-algebra generated by an
embedded copy of $A_1$ and an embedded copy of $A_2$?
\end{question}

The condition above that at least one of
$A_1,$ $A_2,$ have $\!k\!$-dimension $\geq 3$ is needed, for
if both are $\!2\!$-dimensional, say with
bases $\{1,\,b\}$ and $\{1,\,c\},$ then for each $n,$
there are only $2n+1$ alternating words of length $\leq n$ in $b$ and
$c;$ so $A_1\smallcoprod A_2$ has linear growth as a $\!k\!$-algebra.
Hence no subalgebra of a homomorphic
image of that coproduct can have faster than linear growth;
so one cannot, for instance, embed
the free algebra $k\langl x,\,y\rangl$ in such an algebra.
With that case excluded, as in
Question~\ref{Q.A1A2}, $A_1\smallcoprod A_2$ is
easily seen to have exponential growth, and, indeed, to contain
free $\!k\!$-algebras on two generators, which in turn
contain free $\!k\!$--algebras on countably many generators.
If we could show that $A_1\smallcoprod A_2$ had a free
subalgebra on two generators which satisfied the
ideal extension property in $A_1\smallcoprod A_2,$
then we would get a positive
answer to Question~\ref{Q.A1A2} for countable-dimensional~$A_0.$

The case where $k$ is not a field is messier;
in particular, the module-theoretic condition~\eqref{d.A01=k+}
in Theorem~\ref{T.A0A1} definitely cannot be dropped.
For instance, if $k=\Z,$ then the $\!k\!$-algebras
$\Q$ and $\Z+(\Q/\Z),$ where the latter denotes the
result of making $\Q/\Z$ a nonunital $\!\Z\!$-algebra
via the zero multiplication, and then adjoining a unit,
cannot lie in a common unital $\!\Z\!$-algebra,
since a $\!\Q\!$-algebra cannot have additive torsion
-- though $\Z+(\Q/\Z)$ is generated as
a unital $\!\Z\!$-algebra by a module-theoretic image of $\Q.$

\section{Constructions not using generators and relations}\label{S.x<>}

When one wants to establish that
certain relations in an algebra do not entail other relations,
an alternative to directly calculating the consequences
of the relations is to construct an {\em action} of
such an algebra exhibiting the non-equality
(cf.\ discussion in \S11.2 of \cite{<>}).
Proofs of this sort are very convenient when they are available.

In fact, the first result of which I am aware showing that
countably generated rings could be embedded {\em unitally}
in finitely generated rings, the main theorem of~\cite{O+V+W}
(which in fact gives embeddings in $\!2\!$-generated rings),
uses a technique of this sort, formulated in terms
of infinite column-finite matrices over the given ring.
The method is equally applicable to algebras.
(A generalization to topological algebras is given in~\cite{top}.)

Our final result, below shows how one of the results of the
present note, the ``$x\,y^n\,z$'' case of Theorem~\ref{T.easy},
can be given an alternative proof of this sort.

Note first that any $\!k\!$-algebra $A$ can be embedded in
the endomorphism algebra of some $\!k\!$-module (namely, any faithful
$\!A\!$-module, regarded as a $\!k\!$-module), and that by taking a
countably infinite direct sum of copies of such a module, we
get another such $\!k\!$-module
which, moreover, is a countably infinite
direct sum of isomorphic copies of itself.
We can now very quickly prove our result.

\begin{lemma}\label{L.M=(+)}
Let $k$ be a commutative ring, and $M$ a $\!k\!$-module
which is a countably infinite direct sum of isomorphic copies
of itself,
\begin{equation}\label{d.oplus}\mbox{
$M\ =\ \bigoplus_{i=0}^\infty M_i\,.$
}\end{equation}
We shall write elements
of $\r{End}_k(M)$ to the left of their arguments.

Then for every countable family $s_0,\,s_1,\,\dots,\,s_n,\,\dots$ of
members of $\r{End}_k(M),$ there exist $x,\,y,\,z\in\r{End}_k(M)$
satisfying $x\,y^i\,z=s_i,$ i.e.,~\eqref{d.xy^nz}.

Hence for any countably generated $\!k\!$-algebra $A,$
letting $N$ be a faithful $\!A\!$-module, and applying the above
to a direct sum $M$ of
a countably infinite family of copies of $N,$ we recover
the case of Theorem~\ref{T.easy} that uses the
relations~\eqref{d.xy^nz}.
\end{lemma}

\begin{proof}
Given~\eqref{d.oplus},
let $z\in\r{End}_k(M)$ carry $M$ isomorphically to its submodule $M_0,$
and let $y\in\r{End}_k(M)$ take each $M_i$ isomorphically to $M_{i+1}.$
Viewing $y^iz$ as an isomorphism $M\to M_i$ for each $i,$
let $x: M=\bigoplus_{i=0}^\infty M_i\to M$ be the map which
acts on each $M_i$ by $s_i(y^i z)^{-1}.$
Then for each $i$ we have $x\,y^i\,z=s_i,$ as claimed.
Letting $R$ be the $\!k\!$-subalgebra of
$\r{End}_k(M)$ generated by $x,$ $y$ and $z,$ we get the
desired case of Theorem~\ref{T.easy}.
\end{proof}

\section{Acknowledgements}\label{S.ackn}

I am grateful to Ken Goodearl,
Pace Nielsen, Gabriel Sabbagh, Lance Small and Agata Smoktunowicz
for helpful comments on earlier versions of this note,
and pointers to related work.

\end{document}